\documentclass[a4paper]{article}

\usepackage{doi}
\usepackage{amsmath,amsthm,amssymb}
\usepackage{hyperref}
\usepackage[ruled,vlined]{algorithm2e}
\usepackage[margin=1.5cm]{geometry}
\newcommand{\IP}{I\!P}
\newcommand{\LIP}{LI\!P}
\newcommand{\OIP}{O\!I\!P}
\renewcommand{\vec}[1]{\ensuremath{\mathbf{#1}}}
\providecommand{\email}[1]{\href{mailto:#1}{#1}}

\newtheorem{definition}{Definition}
\newtheorem{lemma}{Lemma}
\newtheorem{theorem}{Theorem}
\newtheorem{example}{Example}
\newtheorem{remark}{Remark}

\title{Multi-objective integer programming: Synergistic parallel approaches}

\author{William Pettersson \\ School of Computing Science, University of Glasgow \\ Glasgow G12 8QQ, \textsc{United Kingdom} \\
\email{william@ewpettersson.se} \and
Melih Ozlen \\
School of Science, RMIT University, \\ 
  Victoria~3000, \textsc{Australia}. \\
  \email{melih.ozlen@rmit.edu.au}
}

\begin{document}
\maketitle

\begin{abstract}
  Exactly solving multi-objective integer programming (MOIP) problems is often a very time consuming process, especially for large and complex problems.
  Parallel computing has the potential to significantly reduce the time taken to solve such problems, but only if suitable algorithms are used.
  The first of our new algorithms follows a simple technique that demonstrates impressive performance for its design.
  We then go on to introduce new theory for developing more efficient parallel algorithms.
  The theory utilises elements of the symmetric group to apply a permutation to the objective functions to assign different workloads, and
  applies to algorithms that order the objective functions lexicographically.
  As a result, information and updated bounds can be shared in real time, creating a synergy between threads.
  We design and implement two algorithms that take advantage of such theory.
  To properly analyse the running time of our three algorithms, we compare them against two existing algorithms from the literature, and against using multiple threads within our chosen IP solver, CPLEX.
  This survey of six different parallel algorithms, the first of its kind, demonstrates the advantages of parallel computing.
  Across all problem types tested, our new algorithms are on par with existing algorithms on smaller cases and massively outperform the competition on larger cases.
  These new algorithms, and freely available implementations, allows the investigation of complex MOIP problems with four or more objectives.
\end{abstract}

\section{Introduction}
\subsection{Background}

In multi-objective integer programming (MOIP), one must consider a range of objective functions with the goal of finding all non-dominated objective vectors, sometimes called the Pareto set.
A decision maker can use such a set to compare the various trade-offs that can be made between the objective functions.

The Pareto set can be calculated exactly, or approximated.
Approximation techniques include
 heuristics (or meta-heuristics), swarming (such as those of~\cite{Parsopoulos2002ParticleSwarm,Sethanan2016Swarming}) and evolutionary (see~\cite{cantu2000efficient,Reed2008ParallelEvolutionary}) algorithms.
However, this paper will only consider algorithms which calculate the exact Pareto set, with no omissions or inaccuracies.
For an introduction to multi-objective optimisation in general, see~\cite{Ehrgott2000Survey}, and for a very recent and thorough look at exact MOIP algorithms, focusing on branch and bound algorithms, see~\cite{Przybylski2017}.

This paper looks at parallel multi-objective exact integer optimisation algorithms.
Parallel evolutionary algorithms that find approximate solutions have received significant study in the literature, such as in~\cite{Yu2017Parallel,Pedroso2017Parallel}, but results on exact parallel algorithms for multi-objective problems are not as widespread.
\cite{Lemesre2007PPM} introduce PPM, an algorithm which splits the feasible solution space through a three-stage process.
They first find what they call ``well-distributed solutions'', and then use these solutions to partition the feasible solution space into regions which can be searched in parallel.
\cite{Dhaenens2010KPPM} then extend this work to create K-PPM, which solves problems with more than two objectives.
Being one of the only published algorithms specifically described as a parallel MOIP algorithm, we use it as one of our comparison algorithms.

More recently, \cite{Guo2014ScalingExact} demonstrates parallel improvements through problem-specific information, using specifics of their problems to break down the set of feasible solutions into equitable parts.
We work with generic optimisation problems, and as such the algorithms of Guo et.~al.~cannot solve the problems we test against.

Another method of achieving parallelisation is to iteratively find solutions, and use these solutions to split the objective space into smaller parts.
Each of these smaller parts can then often be independently searched, as mentioned but not implemented in~\cite{boland2015balancedbox}.
This idea, of breaking down the objective space, can also be seen in algorithms that are not necessarily described as parallel, such as in~\cite{Lokman2013,boland2015triangle,Kirlik2014NewAlgorithm,Laumanns2006EfficientAdaptive,Dächert2015}.
We implement V-SPLIT from~\cite{Dächert2015} as our second comparison algorithm, as they prove that it reaches the theoretical best-case in terms of integer problems solved, and they show that it is one of the two faster algorithms in the literature, the second being AIRA by~\cite{Ozlen2014moipaira} which we will also parallelise.
V-SPLIT is only a 3-objective algorithm, unlike all other algorithms discussed, so timing results for V-SPLIT are only available on 3-objectives.

\subsection{Our contribution}
We present three new algorithms.
The first of these calculates the range of values that one of the objective functions may take, and divides this range equally amongst all threads.
Unlike existing algorithms, this partitioning takes place before any searching for solutions.
Timing results show that running time does improve as more threads are used, and that the performance is at least on par with other algorithms in the literature.

We propose a new parallelisation technique for MOIP algorithms.
Threads are given a unique approach to the problem (as determined by an element of the symmetric group $S_n$).
This approach, but limited to the biobjective case where the theory is trivial and there is no synergy, is used in~\cite{OzlenPettersson2016BiObjective} where it just equally partitions the objective space.
We present results which allow the real-time sharing of bounds to all other threads.
Even though each thread is solving the same problem, this sharing creates a synergy between the threads.
As one thread reduces the required computation for a second thread, the second thread will in turn reduce the required running time for the first thread.
This synergy can allow significant performance improvements from parallelisation, especially in problems where the number of objectives is large.
This theory is given as a theoretical background in Section~\ref{sec:sharing} so that it may be used and extended in other parallel algorithms.

We design, implement and test two algorithms based on this theory.
These algorithms are compared to the other state-of-the-art algorithms.
The results show that the new algorithms perform on par on smaller problems, and outperform the existing algorithms on larger problems.
This validates the synergy evident in the theory: we give experimental results that show that the algorithms perform better than existing algorithms across all problems when the thread count is equal to the number of objectives.
Even when thread counts are increased beyond this level, we still see our new algorithms scaling well and outperforming the other existing algorithms on all larger problem instances, and performing at a similar level for the smaller problem instances.

We offer our implementations of these algorithm for further use.
This opens up many new opportunities to solve new problems in optimisation not only where more variables or more objective functions need to be considered, but also in more time-critical scenarios.

\subsection{Paper layout}
The rest of this paper is organised as follows.
Section~\ref{sec:background} gives a background and details the notation we use to describe the symmetric group, symmetries and lexicographically constrained MOIP problems.
In Section~\ref{sec:sharing} we give the theory that demonstrates the sharing of results between still-running threads.
Section~\ref{sec:newalgo} describes our new algorithm, and Section~\ref{sec:impl} discusses some implementation details.
The results of our testing are presented and discussed in Section~\ref{sec:discussion}.
Finally we conclude in Section~\ref{sec:conclusion}.

\section{Background}\label{sec:background}

\subsection{Permutations}

We will use permutations to denote different hierarchical orderings of objectives in lexicographic restrictions of MOIPs.
We will use $S_n$ to denote the symmetric group on the $n$ elements $\{1,2,\ldots,n\}$.
Given a permutation $s \in S_n$, let $s(i)$ be the image of $i$ under $s$.
For example, let $s = (3,2,4,1) \in S_4$. Then $s(1) = 3$, $s(2) = 2$, $s(3) = 4$ and $s(4) = 1$.

\subsection{Multi objective optimisation}

A MOIP is defined as
\begin{align*}
\min \quad  & f_1(\vec x), \ldots, f_n(\vec x) \\
\mbox{~s.t.}  \quad & A\vec x \leq C \\
& \vec x \in \mathbb{Z}^c
\end{align*}

where the matrices $A$ and $C$ are appropriately sized.
For a given feasible solution $\vec x$, we call the associated vector $(f_1(\vec x), f_2(\vec x), \dots, f_n(\vec x))$ the objective vector of $\vec x$.
Where the feasible solution $\vec x$ may not be relevant, we may refer to such a vector as simply an {\em objective vector\/} of the MOIP problem.
Note that there is no guarantee that an objective vector need be optimal, it is simply a vector in the objective space that corresponds to some feasible value of $\vec x$.

\begin{definition}
  A vector $(z_1,z_2,\ldots,z_n)$ is said to dominate vector $(y_1,y_2,\ldots,y_n)$ if
  \begin{enumerate}
    \item $z_k \leq y_k$ for all $k \in \{1,\ldots,n\}$, and
    \item $z_k < y_k$ for at least one $k\in \{1,\ldots,n\}$.
  \end{enumerate}
\end{definition}
The {\em non-dominated\/} objective vectors for a MOIP problem are then exactly the objective vectors which are not dominated. There might be more than one efficient solution in the solution space corresponding to a non-dominated objective vector but we are only concerned with identifying one of those efficient solutions for our purposes.
%Note that there may be multiple feasible solutions $\vec x$, $\vec x'$ such that $(f_1(\vec x),\ldots,f_n(\vec x)) = (f_1(\vec x'),\ldots,f_n(\vec x'))$. 

For conciseness, we assume each objective function is to be minimised, and
for more details on multi-objective optimisation we guide the reader to~\cite{Ehrgott2000Survey}.

%In this definition, $c$ denotes the number of variables in the problem.

The algorithm of~\cite{Ozlen2014moipaira} repeatedly solves ordered, constrained versions of a given MOIP. These consider some of the objectives in a specific order, and also apply upper bounds on the values of some of the objective functions (hence constrained). We will write $\OIP_s^n(k, (a_{s(k+1)},\ldots, a_{s(n)}))$ to refer to such a problem, where
\begin{itemize}
\item $s$ denotes the order in which the objectives are considered;
\item $n$ is the number of objectives in total;
\item $k$ is the number of objectives which have no upper bound; and
\item for $i\in \{k+1,\ldots,n\}$, $a_{s(i)}$ is an upper bound on the value of $f_{s(i)}(\vec x)$ (i.e.~add $f_{s(i)}(\vec x) \leq a_{s(i)}$ as a constraint to the problem).
\end{itemize}

That is, given a MOIP with objective functions $f_1,\ldots,f_n$, we define the $\OIP_s^n(k, (a_{s(k+1)},\ldots, a_{s(n)}))$ as follows:

\begin{align*}
\min \;\; 1\textrm{st} \quad  & f_{s(1)}(\vec x), \ldots, f_{s(k)}(\vec x) \\
\min \;\; 2\textrm{nd} \quad & f_{s(k+1)}(\vec x) \\
\min \;\; 3\textrm{rd} \quad & f_{s(k+2)}(\vec x) \\
 & \quad \vdots \\
\min \;\; \textrm{last} \quad & f_{s(n)}(\vec x) \\
\mbox{~s.t.}  \quad & A\vec x \leq C \\
& f_{s(k+1)}(\vec x) \leq a_{s(k+1)} \\
& f_{s(k+2)}(\vec x) \leq a_{s(k+2)} \\
& \quad \vdots \\
& f_{s(n)}(\vec x) \leq a_{s(n)} \\
& \vec x \in \mathbb{Z}^c
\end{align*}

In these problems, the first $k$ objective functions (from $f_{s(1)}$ to $f_{s(k)}$) are considered in the usual manner for MOIP problems (i.e.~objective vectors which are not dominated in these first $k$ objectives).
Conditional on that, the problem then minimises $f_{s(k+1)}$, then $f_{s(k+2)}$ and so-on.
It is routine to verify that the set of non-dominated objective vectors to this OIP are a subset of the non-dominated objective vectors for the corresponding MOIP.

Note that \cite{Ozlen2014moipaira} only consider the objectives in their natural order (i.e.~under the identity permutation).
Considering different permutations allows us to synergise between parallel threads.

We now give some properties of the set of non-dominated objective vectors for an ordered, constrained lexicographic restriction of a MOIP.

\begin{remark}\label{definition:lips}
  Let $f_1, \ldots, f_n$ be $n$ linear objective functions for a MOIP with feasible solution space $X$, let $s\in S_n$ and let $Y$ be the set of non-dominated objective vectors for the associated constrained ordered problem $\OIP_s^n(k, (a_{s(k+1)}, \ldots, a_{s(n)}))$.
Then for any $y \in Y$

\begin{enumerate}
  \item for any $i$ with $k < i \leq n$, $f_{s(i)}(y) \leq a_{s(i)}$,\label{definition:lips-constraint}
  \item for any $y' \in Y$ with $y' \neq y$, there exists a $j \leq k \text{ s.t. } f_{s(j)}(y) < f_{s(j)}(y')$, and\label{definition:lips-firstk}
  \item for any $x \in X$ with $f_{s(i)}(x) = f_{s(i)}(y)$ for $i \leq k$, there exists a $j \leq n$ such that for all $j' < j$, $f_{s(j')}(y) = f_{s(j')}(x)$ and $f_{s(j)}(y) < f_{s(j)}(x)$.\label{definition:lips-lastn-k}
\end{enumerate}

\end{remark}

In the above, \emph{\ref{definition:lips-constraint}.}~indicates that all non-dominated objective vectors must meet the given bounds.
\emph{\ref{definition:lips-firstk}.}~shows that any non-dominated objective vector cannot be dominated by another in all of the first $k$ objectives.
Lastly, \emph{\ref{definition:lips-lastn-k}.}\ says that if two objective vectors agree in their first $k$ objectives, the final $n-k$ objectives are considered in the order given by the permutation~$s$.

Technically, these constrained lexicographic problems may more accurately be described as a partially constrained, partially lexicographic problems, but this wording gets cumbersome and is skipped in favour of simply constrained lexicographic.

We now give Lemma 4.1 from~\cite{Ozlen2009GeneralApproach}. 
We will later give a restatement of this Lemma (see Theorem \ref{thm:new-recurse}) which (a) takes into consideration different permutations of the objective functions (the original lemma assumes the objective functions are already ordered); and (b) states the implications more explicitly so that the correctness of our algorithm is easier to confirm.

\begin{lemma}[Lemma 4.1 from~\cite{Ozlen2009GeneralApproach}]
If a solution to a $k$ objective problem attains the upper bound on one of these objectives, say $f_i$, then it is also optimal on the $k-1$-objective problem where $f_i$ is no longer considered.
\end{lemma}

\section{Bound sharing}\label{sec:sharing}
This section details our contributions, including both the theory behind our new algorithms as well as our new parallel algorithms.
In these new algorithms, we assign different ordered variants of the input MOIP problem to different threads by means of a permutation $s \in S_n$.
As each thread progresses, it finds non-dominated objective vectors, but also tracks the region of the objective space which it has completely searched.
A thread will share information about such searched regions with other threads, reducing the search space for these other threads.

This section is broken into two subsections: first we give the theory required to prove correctness and later we give the actual algorithm.

\subsection{Theoretical results}
\begin{lemma}\label{lemma:drop-k}
  If $x$ is a non-dominated objective vector for $\OIP^n_s(k-1, (a_{s(k)},\dots,a_{s(n)}))$,
  then $x$ is a non-dominated objective vector for $\OIP^n_s(k, (a_{s(k+1)},\allowbreak \dots,a_{s(n)}))$.
\end{lemma}

This result follows trivially from Remark~\ref{definition:lips}.

Before the next theorem, we define notation for showing that two permutations agree in their ``final'' positions.
\begin{definition}[$s =_a s'$]
  Given two elements $s, s' \in S_n$, if $s(i) = s'(i)$ for all $(n-a) < i \leq n$, we say that $s =_a s'$.
\end{definition}
For example, if $s = (4,1,2,3)$ and $s' = (1,4,2,3)$ then $s =_2 s'$ as both permutations end with ``$\left. 2,3\right)$''.
%Recall that $s =_{n-k} s'$ means that the permutations $s$ and $s'$ agree in their final $n-k$ places.

We now give a slight variant of Lemma 4.1 from~\cite{Ozlen2009GeneralApproach} to allow for different permutations of the objective functions.  

\begin{theorem}\label{thm:new-recurse}
  Let $s, s'$ be elements of $S_n$ with $s =_{n-k} s'$,
  let $Y$ be the set of non-dominated objective vectors for $\OIP_s^n(k-1, (a_{s(k)},\dots,a_{s(n)}))$, let $\hat a = \max\{f_{s(k)}(y) | y\in Y\}$, and
  let $Y'$ be the set of non-dominated objective vectors for $\OIP_{s'}^n(k, (a_{s(k+1)},\dots,\allowbreak a_{s(n)}))$.
Then for any $y' \in Y'$, either
\begin{enumerate}
  \item $y' \in Y$, or 
  \item $f_{s(k)}(y') > a_{s(k)}$, or 
  \item $f_{s(k)}(y') < \hat a$.
\end{enumerate}
\end{theorem}

\begin{proof}
This holds trivially if either $y' \in Y$ or $f_{s(k)}(y') > a_{s(k)}$ so assume $y' \not\in Y$ and $f_{s(k)}(y') \leq a_{s(k)}$.
Then as $y'$ is dominated in $Y$, let $y \in Y$ be an element that dominates $y'$ in $Y$.
Let $i$ be the smallest integer such that $f_{s(i)}(y') < f_{s(i)}(y)$.
There must be such an $i$ as otherwise $y$ would also dominate $y'$ in $Y'$.
We will take cases on $i$.

If $i > k$ then $y$ and $y'$ obtain equal values for the first $k$ objectives.
However, the last $n-k$ objectives are considered in lexicographic order, as determined by $s$.
As $y$ and $y'$ are both feasible for $\OIP^n_{s'}(k, (a_{s(k+1)},\dots,a_{s(n)}))$, there must be some $j$ such that for $j' < j$, $f_{s(j')}(y') = f_{s(j')}(y)$, and $f_{s(j)}(y') < f_{s(j)}(y)$.
However, both $y$ and $y'$ are also feasible for $\OIP^n_s(k-1, (a_{s(k)},\dots,a_{s(n)}))$, and $s =_{n-k} s'$.
Then by the same argument we must find an $i$ such that for $i' < i$, $f_{s(i')}(y) = f_{s(i')}(y)$, and $f_{s(i)}(y) < f_{s(i)}(y')$. This is clearly a contradiction.

If $i < k$, then clearly $y$ cannot dominate $y'$ in $Y$, leading to a contradiction.

Lastly, if $i = k$ then $f_{s(k)}(y') < f_{s(k)}(y) \leq \hat a$.
\end{proof}

Both problems in this theorem do have identical bounds for their final $n-k$ places; this is not a typographical mistake.
This identity between bounds is exactly why threads can share data, and forms the basis of our algorithm.

To make it easier to discuss the sharing of data between threads, we introduce terminology for the region of the objective space which a thread has already searched.
\begin{definition}\label{definition:above}
  Given a problem $P = \OIP^n_s(k, (a_{s(k+1)},\dots,a_{s(n)}))$, we will say that a thread $t$ has {\em found all non-dominated objective vectors above $P$\/} if, for all $j > k$, $t$ has determined all non-dominated objective vectors $x$ to $\OIP^n_s(j, (a_{s(j+1)},\dots,a_{s(n)}))$ which also satisfy $f_{s(j)}(x) > a_{s(j)}$.
\end{definition}

Specifying that a given region of the objective space may be avoided (as all non-dominated objective vectors within it are known) is not always practical.
In comparison, upper bounds may at times be trivially added, as they can simply supersede the upper bound obtained whilst directly solving a constrained lexicographic problem.
However, these upper bounds cannot always be shared.
The following explains exactly when threads are able to share these updated bounds to other threads.
We first present a simplified version of Theorem~\ref{thm:sharing} to help introduce the reader to our approach.

\begin{lemma}\label{lemma:sharing-simple}
  Let $w$ represent a thread which
  \begin{enumerate}
    \item is currently solving 
   $P = \OIP^n_s(n-1,(a_{s(n)}))$, and
 \item has found all solutions above $P$.
  \end{enumerate}
  Then all solutions to the original $\IP$ with $f_{s(n)}(x') \geq a_{s(n)}$ are known.
\end{lemma}

This lemma is also given in~\cite{OzlenPettersson2016BiObjective}, and
the proof of this lemma follows trivially from Definition~\ref{definition:above}.
The lemma says that if a thread is solving $\OIP^n_s(n-1,(a_{s(n)}))$, and has found all solutions above this problem, then any other thread can also ignore any solution $x$ for which $f_{s(n)}(x) > a_{s(n)}$.
Other threads will be using other permutations, so the bound on $f_{s(n)}$ may not be the ``last'' bound for other threads.
This sharing of bounds across many objective functions can create a synergy between threads, where one thread can supply a bound to other threads, which in turn means that those threads also find new bounds faster and these new bounds can be shared back to the original thread.

As mentioned, the above lemma is actually a simplified version of our result, and only shares the bounds on the ``last'' objective function.
Theorem~\ref{thm:sharing} is a more general result, describing exactly when bounds on {\em any\/} objective functions may be shared between threads.
The theorem states that if two threads agree, in both permutations and bounds, in their last $j$ positions, then the bound that a given thread has on objective $n-j$ i.e., the objective just ``before'' the last $j$ can be shared to the other thread, and vice-versa.
Lemma~\ref{lemma:sharing-simple} allows the bound on the last objective to be shared globally i.e., all threads can use the bound.
In comparison Theorem~\ref{thm:sharing} describes the sharing of bounds on any objective, but does place restrictions on which other threads can use this bound.

Theorem~\ref{thm:sharing} specifies that bounds can be shared if two threads agree on their permutations in their last $j$ positions (e.g.~$s =_j s'$). 

We now give two examples of the usage of this sharing, before giving the theorem and proof below.
\begin{example}
  Let $s_1 = (5,1,4,2,3)$ and $s_2 = (1,4,5,2,3)$, and let 
$P_1 = \OIP^5_{s_1}(2,(13,15,18))$ and let $P_2 = \OIP^5_{s_2}(2, (8,15,14))$.
Note that $(5,1,4,\underline{2,3}) =_2 (1,4,5,\underline{2,3})$. That is, $s_1$ and $s_2$ have the same elements in the final two positions of each permutation.
Since $P_1$ and $P_2$ do not have the same bounds on $f_{s_1(5)}$, we have to take $j=0$ in Theorem~\ref{thm:sharing}.
This means that the bound on $f_{s_2(5)}$ from $P_2$ can be shared to $P_1$.
The end result is that thread running $P_1$ can immediately set the bound on $f_{s_1(5)}$ to $14$, so the new version of $P_1$ to be solved is $P'_1 = \OIP^n_{s_1}(2,(13,15,14))$.
\end{example}
This example can be followed on to the next example.
\begin{example}
  Take $s_1 = (5,1,4,2,3)$ and $s_2 = (1,4,5,2,3)$ again, and let 
$P_1 = \OIP^5_{s_1}(2,(13,15,14))$ and let $P_2 = \OIP^5_{s_2}(2, (8,15,14))$.
Again, $s_1 =_2 s_2$.
Now $P_1$ and $P_2$ agree on bounds $a_{s(4)}$ and $a_{s(5)}$, so we take $j=2$ in Theorem~\ref{thm:sharing}.
That means that the bound on objective $f_{s_1(3)}$ from $P'_1$ can be given to the thread solving $P_2$, and the bound on objective $f_{s_2(3)}$ from $P_2$ can be shared to the thread solving $P_1$.
More specifically, as $s_2(3) = 5$, the thread solving $P'_1$ can use $f_5(x) \leq 8$ as an upper bound for any new solutions, and as $s_1(3) = 4$, the thread solving $P_2$ can use $f_4(x) \leq 13$ as an upper bound on for any new solutions.

These upper bounds apply even though $P'_1$ would otherwise not have any bound on $f_5$, and that if such a bound makes the problem infeasible then there are no new solutions to $P'_1$ which have not been found by $P_2$.
\end{example}

We now give the exact theorem and proof.
\begin{theorem}\label{thm:sharing}
  Let $t$ represent a thread which
  \begin{enumerate}
    \item is currently solving 
   $P = \OIP^n_s(k-1,(a_{s(k)},\dots,a_{s(n)}))$, and
 \item has found all solutions above $P$.
  \end{enumerate}
  For any other thread $t'$ which is currently solving $P' = \OIP^n_{s'}(k', (a_{s'(k'+1)},\dots,a_{s'(n)}))$, and for any integer $j \geq 0$ such that all the following hold
  \begin{enumerate}
    \item $j < n-k$,
    \item $j < n-k'$,
    \item $s =_{n-j} s'$, and
    \item $a_{s(n-i)} = a'_{s'(n-i)}$ for $0 \leq i < j$,
  \end{enumerate}
  all solutions $x'$ to $P'$ with $f_{s'(n-j)}(x') \geq a_{s(n-j)}$ are known.
\end{theorem}

\begin{proof}
 Let $x'$ be a solution to $P'$.
 Then by Lemma~\ref{lemma:drop-k} $x'$ is also a solution to $\OIP^n_{s'}(n-j, (a'_{s'(n-j+1)},\dots,\allowbreak a'_{s'(n)}))$.
 However by the conditions in this theorem, this problem is identical to $\OIP^n_s(n-j, (a_{s(n-j+1)},\dots,\allowbreak a_{s(n)}))$, and by the definition of {\em all solutions above $P$}, $t$ has found all solutions to $\OIP^n_s(n-j, (a_{s(n-j+1)},\dots,\allowbreak a_{s(n)}))$ with $f_{s'(n-j)}(x') \geq a_{s(n-j)}$.
 
\end{proof}

We can recover Lemma~\ref{lemma:sharing-simple} from this theorem by letting $j=0$.

\subsection{New algorithms}\label{sec:newalgo}
\subsubsection{Efficient Parallel Projection (EPP)}
The objective space for a MOIP can be envisioned as a $k$-dimensional vector space, where each dimension represents one objective function.
The Efficient Projection Partitioning (EPP) algorithm projects the whole solution space down to one dimension.
Given an objective vector $\vec x = (x_1,\ldots,x_n)$, the projection is achieved through the $n$-th projection map $proj_n(\vec x) = x_n$. 
That is, the objective space is partitioned by only considering the values attained by one objective function.
First we need the following lemma.

\begin{lemma}\label{lemma:upper}
  For $n>1$, an non-dominated objective vector to a MOIP with objective functions $f_1,\ldots,f_n$ that achieves a maximum value on $f_n$ is also a non-dominated objective vector for the MOIP on the same feasible solution space but restricted to the objective functions $f_1,\ldots,f_{n-1}$.
\end{lemma}
This is a well known lemma; for recent proofs see e.g.~Lemma 4.1 in~\cite{Ozlen2009GeneralApproach} or Theorem 2 in~\cite{Dhaenens2010KPPM}.
EPP first calculates all solutions on the first $n-1$ objective functions recursively, with the solution when $n=1$ being trivial.
The set of non-dominated objective vectors on $n-1$ objectives, along with the above Lemma, is used to determine the maximum value on the $n$-th objective; the minimum is found by simple integer programming.
This gives a range of values which $f_n$ can take, which is divided up equally amongst all threads.

\begin{algorithm}[h]\label{algo:dp}
  \DontPrintSemicolon{}
  \KwData{The MOIP $\IP^n$ on $n$ objective functions, and an integer $T$ representing number of threads to use.}
  \KwResult{The set of non-dominated objective vectors.}

  \eIf{$n = 1$}{
    Solve the single-objective problem and return the solution\;
  } {
    Let $X$ be the feasible solution space for this problem.
    Let $\IP^{n-1}$ be this same problem restricted to the first $n-1$ objective functions.
    Calculate the solutions $Y$ to $\IP^{n-1}$ using Algorithm~\ref{algo:dp}\;
    Let \(U = \max \{ f_n(y) | y\in Y\}\)\;
    Let \(L = \min \{ f_n(y) | y\in X\}\)\;
    Let \(step = \frac{U-L}{t}\)\; 
    \For{$ t \in \{0,\dots, t-1\}$} {
      Let \(l = L + t\times step\)\;
      Let \(u = l + step\)\;
      Start a MOIP solver in a new thread to find all solutions $y$ satisfying $l < f_k(y) \leq u$.
    }
    Return the union of the results from all threads started\;

  }
  \caption{The Efficient Projection Partitioning (EPP) algorithm.}
\end{algorithm}

\subsubsection{CLUSTER and SPREAD}
We next introduce the two algorithms CLUSTER and SPREAD, which apply our permutation parallelisation technique to the algorithm of~\cite{Ozlen2014moipaira}.
First, Algorithm~\ref{algo:new} is the algorithm which will initialise and launch all sub-problems.
The initialisation process lets each thread determine which other threads it might be sending information to, and from which threads it might be receiving information.
Each parallel thread will be running Algorithm~\ref{algo:sub}, where new solutions will be found and new bounds will be calculated and shared.

In Algorithm~\ref{algo:new}, the method for selecting permutations is not specified.
We devise two ways of selecting permutations, which in turn create the two algorithms which we call CLUSTER and SPREAD.
CLUSTER assigns permutations to maximise $i$ where $s =_i s'$ for all selected $s$ and $s'$.
For instance, we could assign $(1,2,3,4,5)$, $(2,1,3,4,5)$, $(1,3,2,4,5)$, $(3,1,2,4,5)$, $(2,3,1,4,5)$ and $(3,2,1,4,5)$ to six threads solving a 5-objective problem.
In other words, all of these have 4 and 5 as their final two elements, and thus can share updates on their third objectives.
These six threads would be sharing updated bounds on deeper levels of the recursion, meaning the algorithms will share bounds more often.
This reduces the time between the determination of a new bound, and when threads can use the new bound, potentially minimising the amount of redundant work completed.
As a downside, though, these bounds might not be shareable with all other threads.

The second option, which we call SPREAD, assigns permutations to minimise $i$ where $s =_i s'$ for all selected $s$ and $s'$.
For instance, this could mean assigning $(1,2,3,4,5)$, $(2,3,4,5,1)$, $(3,4,5,1,2)$, $(4,5,1,2,3)$, $(5,1,2,3,4)$ and $(2,3,4,1,5)$ to six threads solving a 5-objective problem.
All five objectives occur as a ``fifth'' objective in some thread, so every thread will be able to update bounds on every objective.
The sharing of these bounds would mainly happen at the higher level of recursion, i.e.~not as often, but the bounds will be shared to more threads.
We discuss in Section~\ref{sec:discussion} how different selection methods can impact the running time of the algorithm.

\begin{algorithm}\label{algo:new}
  \DontPrintSemicolon
  \KwData{The problem $\IP^n$, and $t$ representing the number of threads to use}
  \KwResult{$ND$: The set of non-dominated objective vectors}
  \Begin{
    Let $L$ be a list of thread details, to be used to tell threads where they are sharing information\;
    \For{$i\in\{1,\dots,t\}$}{
      Create a thread $w$\;
      Select a permutation $s \in S_n$\;
      Create the problem $P_t = \OIP^n_s(n, ())$\;
      Store the details of this thread in $L$\;
    }
    \For{Each element $l$ in $L$}{
      Launch Algorithm~\ref{algo:sub} with the corresponding problem $P_t = \OIP^n_s(n, ())$ taken from $l$, as well as a copy of $L$\;
    }
    Wait for all threads to complete\;
    Let $ND = \bigcup_t \{\text{ solutions to } P_t\}$\;
  }
  \caption{Our new parallel algorithm. This particular algorithm will set up each thread with an appropriately selected permutation $s$. The actual work is done in Algorithm~\ref{algo:sub} which is called from this algorithm.}
\end{algorithm}

\begin{algorithm}\label{algo:sub}
  \DontPrintSemicolon
  \KwData{The problem $\OIP^n_s(k, (a_{s(k+1)},\dots,a_{s(n)}))$, and the details of all other threads solving the same original problem $\IP^n$}
  \KwResult{$ND_k$, the set of non-dominated objective vectors}
  \Begin{
    Set $ND_k = \emptyset$.\;
    \eIf{a relaxation of this problem is already solved \KwSty{and} each solution to said relaxation satisfies the current bounds}{
      Let $ND_k$ be this set of solutions\;
    }{
      \eIf{ $k=1$} {
        Solve the single-objective problem.\;
        \If{the problem is feasible, with solution $x$} {
          Set $ND_k = \{x\}$\;
        }
      }{ %else
        Let $a_{s(k)} = \infty$\;
        From $\OIP^n_s(k, (a_{s(k+1)},\dots,a_{s(n)}))$, create $P = \OIP^n_s(k-1, (a_{s(k)}, a_{s(k+1)}, \dots, a_{s(n)}))$.\;
        Solve $P$ using this algorithm\;
        \While{$P$ is feasible}{
          Let $Y$ be the solutions to $P$, as determined by this algorithm\;
          Let $ND_k = ND_k \cup Y$\;
          Let $a_{s(k)} = \max \left\{ a_{s(k)}, \max\{f_k(x)| x \in Y\} \right\} $\;
          \For{Each thread $w$ with corresponding permutation $s'$}{
            Use Theorem~\ref{thm:sharing} to update the bounds on $P$\;
            \If{$s =_{n-k} s'$ \textbf{and} $w$ has found a higher value for $a_{s(k)}$}{
              Update $a_{s(k)}$\;
            }
          }
          Update $P$ with the new value of $a_{s(k)}$\;
          Solve $P$ using this algorithm\;
        }
      }
    }
  }
  \caption{This algorithm calculates actual solutions to the problem at hand.
    %The set-up for this algorithm is performed by Algorithm~\ref{algo:new}.
}
\end{algorithm}

In Theorem~\ref{thm:new-recurse}, we define $\hat a$ to be the maximum value of $f_{s(k)}(y)$ for any solution $y$.
To allow each thread to apply Theorem~\ref{thm:new-recurse} we must therefore share not only updated bounds, but the maximum value of $f_{s(k)}$ that is attained.
Theorem~\ref{thm:new-recurse} then trivially verifies correctness of this algorithm.

\section{Implementation details}\label{sec:impl}

The implementation of this new algorithm is based on AIRA as used in~\cite{OzlenPettersson2016BiObjective}.
The availability of the source code sped up the implementation process.
The implementation is in C++11, and uses the shared memory and threading features of the Standard Template Library to handle all thread creation and data sharing.
The code is published on Github (see~\cite{moip_aira}), and test cases are also provided (see~\cite{figshare:3obj,figshare:4obj}) for others to utilise.
Our implementations, including the comparison algorithms from the following section, were verified by comparing results between the various implementations against the known results taken from~\cite{Ozlen2014moipaira}.

\subsection{Comparison algorithms}
We compare the running time of both variants of our algorithm against the following algorithms.
AIRA by~\cite{Ozlen2014moipaira} is a state of the art MOIP solver, which uses CPLEX as an single-objective IP solver internally.
In recent results, such as~\cite{Dächert2015}, AIRA was shown to be one of two algorithms to outperform all others, with the second being V-SPLIT which we discuss below.
One very simple method of parallelising AIRA, or indeed most MOIP algorithms, is to allow the IP solver to utilise more threads.
This technique was also seen in~\cite{boland2015balancedbox}, and we call such an improvement CPLEX.
We do not expect that CPLEX will be competitive in this setting, as CPLEX would not understand the whole MOIP problem.
Instead, these numbers display the significant improvements that can be achieved by designing algorithms to suit parallelisation.

The second comparison algorithm is K-PPM, as described in~\cite{Dhaenens2010KPPM}.
This one of the only recent general MOIP algorithms that is specifically described as being parallel.
K-PPM utilises a 3-step process to create a number of sub-problems.
The first phase calculates the ideal and nadir points of the given problem by recursively solving smaller problems.
This does have a cost, one that the authors of K-PPM discussed in~\cite{Dhaenens2010KPPM}.
We chose to implement K-PPM exactly as they described it, so as to not complicate the results.
These ideal and nadir points are used in the second phase to calculate some well-distributed solutions, which in turn are used to partition the solution space.
This partitioning of the solution space creates a number of sub-problems.
Each of these can be solved in parallel by either a generic serial MOIP solver, or potentially a specialised solver.
We chose to use AIRA as the generic MOIP solver for K-PPM, as it is a modern and open source generic MOIP solver, and being very similar to Algorithm~\ref{algo:new} this will reduce any differences caused by the MOIP solver chosen and will instead allow us to highlight the differences due to our new parallelisation technique.

The final comparison algorithm is V-SPLIT, as described in~\cite{Dächert2015}.
V-SPLIT follows a common approach in MOIP algorithms, where new non-dominated objective vectors are found and then used to reduce and/or partition the objective space.
Other recent algorithms that take such an approach can be found in~\cite{boland2015triangle,boland2015balancedbox,Kirlik2014NewAlgorithm,Laumanns2006EfficientAdaptive}.
V-SPLIT, as given in~\cite{Dächert2015}, is only suitable for 3-objective problems.
However we still use V-SPLIT as the only recent comparison between sequential exact MOIP algorithms (\cite{Dächert2015}) showed that V-SPLIT and AIRA are the two leading competitors in the field.
Two variants of V-SPLIT were presented in~\cite{Dächert2015}, using either a weighted Tchebycheff or an $\epsilon$-constraint method to find individual solutions.
While the $\epsilon$-constraint method was faster, the authors of~\cite{Dächert2015} also point out that this method does not allow the arbitrary selection of the next sub-space to be searched.
As a parallel implementation of V-SPLIT would require the simultaneous searching of multiple sub-spaces, we implement a parallel version of the weighted Tchebycheff V-SPLIT algorithm.
Note that V-SPLIT is specifically designed for 3-objective problems, and as such we can only test it on 3-objective problems.

\subsection{Instance generation}\label{sec:problemgen}
Easily accessible sets of MOIP instances suitable for benchmarking are rare in the literature.
The most commonly referenced ``modern'' set appears to be from~\cite{Laumanns2006EfficientAdaptive}, however the website mentioned in their paper no longer provides the actual instances.
Even when these can be found, they were initially used over 10 years ago, and in our experimentation we find that some instances are trivial to solve simply because computer hardware and IP solvers have improved.

We therefore generate our own set of benchmark instances, using similar techniques to those of~\cite{Laumanns2006EfficientAdaptive} (for knapsack problems), \cite{Przybylski2010TwoPhase} (for assignment problems), and~\cite{Ozpeynirci2010ExactSupported} (for traveling salesman problems).
A knapsack instance is generated by randomly assigned an integer weight (uniformly at random in the range $\{60,\ldots,100\}$) to each of $n$ items. The upper bound on the total weight of the selected items is set to be half of the total weight of all items.
Each objective function is chosen in a similar manner, with the coefficients for each item drawn uniformly at random from the range $\{[60,\ldots,100\}$.
Assignment problems are generated in the manner of~\cite{Przybylski2010TwoPhase}, with objective function coefficients drawn uniformly at random from $\{0,\ldots,20\}$.
We also generate instances of the traveling salesman problem as per~\cite{Ozpeynirci2010ExactSupported}. We place cities on a $1000 \times 1000$ plane by assigning integer coordinates to cities, and round the Euclidean distance between any two cities to an integer value.

For each type of problem and each size parameter, we generate five instances. All of these test instances, including some which we could not solve, are provided for further research (\cite{figshare:3obj,figshare:4obj}) and the generator is available at \url{https://github.com/WPettersson/ProblemGenerator}. Specifics on these instances can be seen in Tables \ref{table:3obj-probs} and \ref{table:4obj-probs}. Note that the 3-objective and 4-objective problems were generated independently.

\begin{table}\centering\small
\begin{tabular}{c|ccccc}
Problem & Integer variables & Binary variables & Constraints & Objectives & Non-dominated solutions\\
\hline
AP10 & 0 & 100 & 20 & 3 & 207 \\
AP15 & 0 & 225 & 30 & 3 & 512 \\
AP20 & 0 & 400 & 40 & 3 & 1515 \\
AP25 & 0 & 625 & 50 & 3 & 3333 \\
AP30 & 0 & 900 & 60 & 3 & 6900 \\
AP40 & 0 & 1600 & 80 & 3 & 13403 \\
KP50 & 0 & 50 & 1 & 3 & 426 \\
KP75 & 0 & 75 & 1 & 3 & 1166 \\
KP100 & 0 & 100 & 1 & 3 & 1483 \\
KP125 & 0 & 125 & 1 & 3 & 3058 \\
KP150 & 0 & 150 & 1 & 3 & 4069 \\
KP200 & 0 & 200 & 1 & 3 & 13058\\
TSP10 & 10 & 90 & 110 & 3 & 250 \\
TSP12 & 12 & 132 & 156 & 3 & 384 \\
TSP15 & 15 & 210 & 240 & 3 & 1202 \\
TSP20 & 20 & 380 & 420 & 3 & 3237 \\
TSP30 & 30 & 870 & 930 & 3 & 11651 
\end{tabular}
\caption{Statistics on the randomly generated 3-objective problems}\label{table:3obj-probs}
\end{table}

\begin{table}\centering\small
\begin{tabular}{c|ccccc}
Problem & Integer variables & Binary variables & Constraints & Objectives & Non-dominated solutions \\
\hline
AP05 & 0 & 25 & 10 & 4 & 23 \\
AP08 & 0 & 64 & 16 & 4 & 269 \\
AP10 & 0 & 100 & 20 & 4 & 679 \\
AP11 & 0 & 121 & 22 & 4 & 2672 \\
AP12 & 0 & 144 & 24 & 4 & 1665 \\
AP15 & 0 & 225 & 30 & 4 & 15535 \\
AP20 & 0 & 400 & 40 & 4 & 28274 \\
KP20 & 0 & 20 & 1 & 4 & 43 \\
KP40 & 0 & 40 & 1 & 4 & 632 \\
KP60 & 0 & 60 & 1 & 4 & 2756 \\
KP80 & 0 & 80 & 1 & 4 & 3733 \\
TSP06 & 6 & 30 & 42 & 4 & 50 \\
TSP08 & 8 & 56 & 72 & 4 & 253 \\
TSP10 & 10 & 90 & 110 & 4 & 683 \\
TSP12 & 12 & 132 & 156 & 4 & 3036 \\
TSP15 & 15 & 210 & 240 & 4 & 8489
\end{tabular}
\caption{Statistics on the randomly generated 4-objective problems}\label{table:4obj-probs}
\end{table}

\subsection{Execution environment}

We ran our implementation on the Raijin, a supercomputer run by the National Computing Infrastructure in Australia, which utilises Intel Xeon E5--2670 CPUs running at 2.60GHz.
Our code was compiled with GCC 4.9.0, and we used CPLEX 12.7.0 as our single objective IP solver, and settings for CPLEX were left at their default, except to limit the number of threads which CPLEX could internally spawn, and also enable deterministic parallelism in such cases where CPLEX would spawn multiple threads.
We ran each algorithm over the 3-objective and 4-objective problems described in Section \ref{sec:problemgen} (for 2-objective problems, our new algorithms reduce to the much simpler case with no synergy as given in \cite{OzlenPettersson2016BiObjective}, which also includes experimental results).
The aim of this computational study is to compare the scalability of our new parallel algorithms to existing parallel algorithms from the literature.
We tested each algorithm with 3 and 6 threads for 3-objective problems, and 4, 8, and 12 threads for 4-objective problems.
For the 3-objective problems, SPREAD and CLUSTER with 6 threads are equivalent, as there are only 6 permutations in $S_3$.
We also ran the non-parallel AIRA on all of these problems to see whether the parallel algorithms actually improved the running time.
We have excluded the running times for problems which were solved very quickly (under one second) as well as problems which did not complete in the given time limits (48 hours) across all algorithms.
These harder instances are provided in \cite{figshare:3obj,figshare:4obj} to challenge further research in this area.
Additionally, we do not include running times for CPLEX on 12 threads as it had already showed minimal improvement moving to 8 threads.

\begin{table}\small
\begin{tabular}{c|c|cc|cc|cc|cc|c|cc}
& AIRA & \multicolumn{2}{c|}{CPLEX}& \multicolumn{2}{c|}{K-PPM}& \multicolumn{2}{c|}{EPP}& \multicolumn{2}{c|}{V-SPLIT}& \multicolumn{1}{c|}{SPREAD}& \multicolumn{2}{c}{CLUSTER}\\
Threads & 1 & 3 & 6 & 3 & 6 & 3 & 6 & 3 & 6 & 3 & 3 & 6\\
\hline
AP10 & 18.61 & 11.09 & 11.01 & 15.66 & 9.53 & 9.39 & 6.37 & 9.24 & 4.56 & 7.92 & 8.00 & 4.93 \\ 
AP15 & 100 & 61.82 & 58.74 & 72.56 & 48.28 & 47.94 & 29.52 & 45.06 & 22.21 & 42.78 & 43.40 & 24.76 \\ 
AP20 & 405 & 272 & 251 & 290 & 204 & 202 & 122 & 183 & 90.51 & 176 & 181 & 93.15 \\ 
AP25 & 1085 & 755 & 681 & 792 & 566 & 535 & 315 & 524 & 260 & 468 & 479 & 233 \\ 
AP30 & 2706 & 2743 & 1713 & 1818 & 1399 & 1396 & 801 & 1463 & 713 & 1185 & 1189 & 582 \\ 
AP40 & 7051 & 6330 & 4624 & 4998 & 4056 & 4016 & 2183 & 4879 & 2445 & 3104 & 3148 & 1530 \\ 

KP50 & 38.62 & 38.64 & 38.90 & 34.98 & 21.69 & 21.78 & 13.62 & 14.23 & 6.99 & 17.37 & 17.98 & 11.31 \\ 
KP75 & 237 & 186 & 179 & 222 & 148 & 137 & 76.44 & 123 & 59.49 & 104 & 104 & 62.80 \\ 
KP100 & 667 & 491 & 461 & 572 & 454 & 412 & 259 & 472 & 219 & 309 & 309 & 177 \\ 
KP125 & 2063 & 1255 & 1102 & 1644 & 1189 & 1151 & 727 & 1488 & 727 & 866 & 869 & 524 \\ 
KP150 & 3338 & 4223 & 1883 & 2524 & 2284 & 1816 & 1099 & 3452 & 1585 & 1467 & 1503 & 802 \\ 
KP200 & 18643 & 9118 & 8331 & 13176 & 11087 & 9946 & 6897 & 11408 & 5790 & 8120 & 8273 & 3788 \\ 

TSP10 & 37.34 & 33.54 & 24.05 & 36.26 & 21.96 & 21.38 & 12.17 & 15.01 & 7.53 & 15.84 & 15.68 & 12.30 \\ 
TSP12 & 68.00 & 59.52 & 44.21 & 60.69 & 34.46 & 36.41 & 20.59 & 28.23 & 14.56 & 29.30 & 29.02 & 22.18 \\ 
TSP15 & 443 & 365 & 294 & 313 & 173 & 233 & 137 & 199 & 102 & 187 & 181 & 138 \\ 
TSP20 & 2201 & 1663 & 1418 & 1399 & 931 & 1203 & 691 & 1677 & 818 & 896 & 864 & 670 \\ 
TSP30 & 18067 & 11700 & 10044 & 10310 & 6713 & 8567 & 5060 & 34551 & 17225 & 7284 & 6990 & 5036 \\ 

\end{tabular}
\caption{Running times of each algorithm and various thread counts on a number of assignment, knapsack and traveling salesman problems with three objectives. As CLUSTER and SPREAD are identical when using 6 threads, we only list the running times under the column CLUSTER.}\label{table:3obj-times}
\end{table}

\begin{table}\small
\begin{tabular}{c|c|c|c|cc|cc|c|cc}
& Solutions & AIRA & \multicolumn{1}{c|}{K-PPM}& \multicolumn{2}{c|}{EPP}& \multicolumn{2}{c|}{V-SPLIT}& \multicolumn{1}{c|}{SPREAD}& \multicolumn{2}{c}{CLUSTER}\\
Threads & $|ND|$ & 1 & 3/6 & 3 & 6 & 3 & 6 & 3 & 3 & 6\\
\hline
AP10 & 207 & 1021 & 2190 & 1153 & 1241 & 657 & 666 & 1444 & 1444 & 1662 \\ 
AP15 & 512 & 2434 & 4782 & 2652 & 2841 & 1601 & 1626 & 3459 & 3465 & 3951 \\ 
AP20 & 1515 & 7017 & 12900 & 7282 & 7552 & 4741 & 4834 & 9730 & 9727 & 10479 \\ 
AP25 & 3333 & 14914 & 24565 & 15365 & 15736 & 10417 & 10676 & 19685 & 19721 & 19506 \\ 
AP30 & 6900 & 29487 & 47653 & 29933 & 30362 & 21462 & 21826 & 40487 & 40526 & 38475 \\ 
AP40 & 13403 & 52350 & 80497 & 52921 & 53495 & 41562 & 42245 & 74334 & 74423 & 72396 \\ 
KP50 & 426 & 2352 & 5285 & 2526 & 2670 & 893 & 955 & 3474 & 3464 & 3693 \\ 
KP75 & 1166 & 6252 & 14240 & 6490 & 6724 & 2321 & 2432 & 8296 & 8278 & 9687 \\ 
KP100 & 1483 & 7749 & 15197 & 8068 & 8314 & 3149 & 3046 & 10698 & 10719 & 11886 \\ 
KP125 & 3058 & 15289 & 27340 & 15742 & 16164 & 5933 & 6081 & 22983 & 22928 & 22521 \\ 
KP150 & 4069 & 20786 & 39899 & 21515 & 22049 & 6894 & 7325 & 31819 & 31931 & 30387 \\ 
KP200 & 13058 & 67857 & 118222 & 68933 & 70035 & 18142 & 19199 & 93682 & 93902 & 88758 \\ 
TSP10 & 250 & 1448 & 3648 & 1601 & 1744 & 794 & 818 & 2033 & 2028 & 2636 \\ 
TSP12 & 384 & 2200 & 5184 & 2388 & 2565 & 1226 & 1257 & 3138 & 3136 & 4071 \\ 
TSP15 & 1202 & 7090 & 14268 & 7403 & 7691 & 3814 & 3932 & 9722 & 9730 & 12489 \\ 
TSP20 & 3237 & 19093 & 33644 & 19755 & 20382 & 10198 & 10489 & 25201 & 25214 & 33282 \\ 
TSP30 & 11651 & 68502 & 119519 & 69639 & 70711 & 35952 & 36802 & 92422 & 92402 & 115955 \\

\end{tabular}
\caption{The number of non-dominated solutions for each problem, and the number of single-objective IPs solved by each type of algorithm for each problem. Note that K-PPM always solves the same number of IPs, using more threads means that more are solved simultaneously. SPREAD and CLUSTER are the same on 6 threads (as $|S_3|=6$), so the statistics for SPREAD on 6 threads are not shown.}\label{table:3obj-stats}
\end{table}
\begin{table}\scriptsize
\begin{tabular}{c|c|cc|ccc|ccc|ccc|ccc}
& AIRA & \multicolumn{2}{c|}{CPLEX}& \multicolumn{3}{c|}{K-PPM}& \multicolumn{3}{c|}{EPP}& \multicolumn{3}{c|}{CLUSTER}& \multicolumn{3}{c}{SPREAD}\\
Threads & 1 & 4 & 8 & 4 & 8 & 12 & 4 & 8 & 12 & 4 & 8 & 12 & 4 & 8 & 12\\
\hline
AP05 & 0.98 & 0.94 & 0.97 & 4.35 & 2.38 & 1.74 & 0.59 & 0.43 & 0.39 & 0.97 & 0.64 & 0.67 & 0.46 & 0.46 & 0.43 \\ 
AP08 & 67.66 & 37.72 & 36.60 & 121 & 61.52 & 45.35 & 31.72 & 22.12 & 16.11 & 48.80 & 28.64 & 26.82 & 24.27 & 22.45 & 17.54 \\ 
AP10 & 532 & 242 & 230 & 732 & 328 & 241 & 240 & 125 & 97.55 & 322 & 182 & 153 & 185 & 147 & 109 \\ 
AP11 & 1095 & 513 & 476 & 1411 & 642 & 471 & 539 & 279 & 201 & 636 & 353 & 285 & 397 & 285 & 212 \\ 
AP12 & 1284 & 621 & 571 & 1567 & 732 & 533 & 634 & 325 & 246 & 744 & 415 & 340 & 450 & 332 & 251 \\ 
AP15 & 7018 & 3742 & 3061 & 7149 & 3655 & 2757 & 3718 & 1779 & 1287 & 3661 & 1983 & 1560 & 2572 & 1726 & 1239 \\ 
AP20 & 51505 & 31234 & 22761 & 29137 & 17060 & 14499 & 24788 & 12834 & 9126 & 18392 & 11638 & 8841 & 16980 & 11285 & 7595 \\ 
KP20 & 3.56 & 2.96 & 2.99 & 12.69 & 6.54 & 4.81 & 2.38 & 1.61 & 1.39 & 3.85 & 2.59 & 2.71 & 2.00 & 1.95 & 1.61 \\ 
KP40 & 177 & 169 & 153 & 275 & 148 & 111 & 116 & 63.24 & 56.16 & 147 & 90.76 & 84.68 & 73.85 & 63.06 & 51.63 \\ 
KP60 & 2529 & 2223 & 2242 & 2329 & 1205 & 1046 & 1564 & 776 & 600 & 1564 & 861 & 689 & 840 & 651 & 487 \\ 
KP80 & 17265 & 6774 & 8429 & 15095 & 8326 & 6617 & 10082 & 7251 & 4089 & 10413 & 5791 & 4528 & 5710 & 4332 & 3200 \\ 
TSP06 & 5.04 & 4.65 & 403.59 & 18.66 & 10.00 & 7.18 & 3.04 & 2.27 & 1.94 & 2.25 & 2.48 & 2.17 & 2.15 & 2.47 & 2.16 \\ 
TSP08 & 74.89 & 59.05 & 2478.3 & 160.70 & 88.01 & 63.91 & 40.33 & 27.97 & 24.29 & 30.03 & 32.43 & 27.79 & 29.81 & 32.53 & 27.87 \\ 
TSP10 & 426 & 343 & 5341 & 692 & 388 & 313 & 219 & 129 & 94.87 & 159 & 167 & 145 & 159 & 168 & 140 \\ 
TSP12 & 6555 & 5180 & 20530 & 5032 & 3437 & 3331 & 4678 & 2753 & 2091 & 1693 & 1711 & 1305 & 1689 & 1714 & 1265 \\ 
TSP15 & 56759 & 47991 & 54186 & 24803 & 13670 & 11183 & 57546 & 28546 & 21367 & 9134 & 7812 & 6165 & 8960 & 7821 & 6077 \\ 

\end{tabular}
\caption{Running times of each algorithm and various thread counts on a number of assignment, knapsack and traveling salesman problems with four objectives.}\label{table:4obj-times}
\end{table}

\begin{table}\scriptsize
\begin{tabular}{c|c|c|c|ccc|ccc|ccc}
& Solutions & AIRA & \multicolumn{1}{c|}{K-PPM}& \multicolumn{3}{c|}{EPP}& \multicolumn{3}{c|}{SPREAD}& \multicolumn{3}{c}{CLUSTER}\\
Threads & $|ND|$ & 1 & 4/8/12 & 4 & 8 & 12 & 4 & 8 & 12 & 4 & 8 & 12 \\
\hline

AP05 & 23 & 299 & 7668 & 468 & 627 & 847 & 808 & 939 & 1359 & 520 & 1006 & 1256 \\ 
AP08 & 269 & 3957 & 36917 & 5715 & 6931 & 8563 & 11157 & 13972 & 18886 & 7176 & 12845 & 14292 \\ 
AP10 & 679 & 9756 & 72036 & 12164 & 14228 & 16293 & 24562 & 29538 & 38643 & 15662 & 26046 & 30105 \\ 
AP11 & 2672 & 37950 & 220255 & 42253 & 45976 & 49841 & 79755 & 91593 & 112093 & 56873 & 83973 & 96842 \\ 
AP12 & 1665 & 21616 & 106153 & 26760 & 30055 & 35150 & 55763 & 63829 & 79325 & 35141 & 55965 & 62788 \\ 
AP15 & 15535 & 203178 & 754893 & 215648 & 227056 & 236495 & 362804 & 391911 & 455529 & 298430 & 399947 & 427234 \\ 
AP20 & 28274 & 335840 & 771710 & 356185 & 375345 & 393145 & 959667 & 664263 & 777579 & 487597 & 641618 & 678781 \\ 
KP20 & 43 & 577 & 7706 & 1057 & 1468 & 1921 & 1854 & 2483 & 3697 & 1156 & 2314 & 2681 \\ 
KP40 & 632 & 10210 & 62758 & 11948 & 13584 & 14947 & 28079 & 31834 & 42649 & 16576 & 27895 & 33640 \\ 
KP60 & 2756 & 39004 & 200124 & 46324 & 51762 & 57568 & 96652 & 110956 & 139901 & 62906 & 100714 & 112003 \\ 
KP80 & 3733 & 52946 & 291961 & 63730 & 72539 & 81240 & 134036 & 164911 & 203856 & 83007 & 121448 & 154670 \\ 
TSP06 & 50 & 758 & 11408 & 1176 & 1520 & 1846 & 1427 & 2713 & 3275 & 1416 & 2708 & 3293 \\ 
TSP08 & 253 & 4927 & 44179 & 6598 & 8103 & 9497 & 8198 & 15337 & 17822 & 8212 & 15339 & 17787 \\ 
TSP10 & 683 & 14249 & 96713 & 16619 & 18660 & 20739 & 22263 & 41123 & 48686 & 22259 & 41202 & 48502 \\ 
TSP12 & 3036 & 65657 & 318298 & 72249 & 78245 & 83711 & 96426 & 181954 & 205593 & 96405 & 181940 & 205976 \\ 
TSP15 & 8489 & 190897 & 877615 & 209635 & 226972 & 243309 & 269413 & 466893 & 548163 & 269360 & 466879 & 547828 \\

\end{tabular}
\caption{The number of non-dominated solutions for each problem, and the number of single-objective IPs solved by each type of algorithm for each problem. Note that K-PPM always solves the same number of IPs, using more threads means that more are solved simultaneously.}\label{table:4obj-stats}
\end{table}

\subsection{Discussion}\label{sec:discussion}
For our 3-objective tests (respectively 4-objective tests), we show the average running time across all 5 instances for each size parameter in Table \ref{table:3obj-times} (respectively Table \ref{table:4obj-times}) while Table \ref{table:3obj-stats} (respectively Table \ref{table:4obj-stats}) shows the average  number of solutions to each problem, as well as the average number of single-objective IPs solved by each algorithm.
The AP, KP or TSP in the name of each test refers to the problem type (either assignment problem, knapsack problem or traveling salesman problem) and the number refers to the respective number of objects in each problem (total number of agents for assignment problems, number of objects for knapsack problems and number of cities for the traveling salesman problems).
First we note that as $|S_3| = 6$, the algorithms SPREAD and CLUSTER are identical when using six threads, so we omit the column for SPREAD on six threads as it is identical to the column for CLUSTER on six threads.
Secondly we point out that K-PPM always solves the same number of IPs independently of the number of threads used.

Looking at the running times, we see that CPLEX does gain some, but not much, improvement with parallelisation.
In particular, going from one to three (or four) threads does seem slightly useful, but stepping beyond this is less effective, especially for the smaller problems.
This is consistent with expectation for this approach.

K-PPM does improve as more threads are introduced, performing better than CPLEX.
However, it is in turn beaten by the remaining four algorithms. K-PPM solves more single-objective IPs than all other algorithms, which may explain why it does not perform as well.

V-SPLIT finds all solutions whilst solving the fewest number of IPs. We note that V-SPLIT does solve slightly more IPs when using 6 threads as compared to 3, which would seem to be at odds with the theory in~\cite{Dächert2015}, however as we implemented a parallel version of V-SPLIT some of the theoretical results will no longer hold for our implementation.

For smaller problems, we see that V-SPLIT marginally outperforms EPP.
In~\cite{Dächert2015} the largest problem solved was a knapsack problem with 50 objects, equivalent to our smallest knapsack problem.
Our timing results do then correlate with those from~\cite{Dächert2015}.
However as the problems get bigger, EPP performs better than VSPLIT.
It is surprising that EPP should be competitive, as initially EPP seems like a very basic parallel algorithm, and EPP does solve at times significantly more IPs than V-SPLIT.
This may be explained by the new constraints added by V-SPLIT to search only a specified region. V-SPLIT must add three lower bounds and three upper bounds for each region, while EPP is only required to add one lower bound and one upper bound.

Both CLUSTER and SPREAD display performance improvements as more threads are utilised, and for the larger problems solved both outperform all other algorithms.
On 3-objective problems, and with three threads, both CLUSTER and SPREAD appear to perform similarily.
This is not too surprising, as there are only 6 possible permutations to choose from, so the difference is not as evident.
Clearly as $|S_3|=6$ SPREAD and CLUSTER are identical on 6 threads and so only one column is shown.

The difference between CLUSTER and SPREAD becomes more evident on 4-objective problems, where we only choose 4 (or 8 or 12) of the possible 24 permutations.
On a significant proportion of our test cases we see that SPREAD beats EPP, but EPP beats CLUSTER.
An analysis of the running of CLUSTER and SPREAD gives one possible explanation for the difference in running times between the two.
When solving a biobjective problem (such as $\OIP(2,(a_3,\dots,a_n))$, variants of which get solved repeatedly), the algorithms often only find one or two new solutions i.e., solutions which aren't found via a relaxation.
However, if two threads are attempting to solve a biobjective problem from two different permutations, there is only ever a performance increase if they can solve for different solutions, which requires at least 3 new solutions in each biobjective problem.
This was very rare in our randomly-generated problems.
It is definitely plausible that there exist problems where each new biobjective problem has numerous solutions, and in these cases we believe that the CLUSTER algorithm may perform better, but we are not aware of any research into finding such problems.

\section{Conclusion}\label{sec:conclusion}

We demonstrate a new paradigm for approaching parallelisation in multi-objective optimisation problems.
By utilising different permutations of objective functions, our new theory presents many different directions from which a MOIP problem can be solved.
This allows parallel algorithms to start searching almost immediately for solutions to the problem, rather than spending time trying to find an equitable split of the search space.
The threads are also able to communicate in real time, and this communication creates a synergy where each thread can reduce the running time of all other threads, which in turn can speed up the first thread.

We give the first comparative look at the running time of exact MOIP algorithms in parallel settings.
This shows that even some seemingly sequential algorithms such as V-SPLIT can benefit from parallelisation.
We also introduce three of our own new parallel algorithms, along with implementations.
All three new algorithms perform competitively on the smaller test cases, and on larger test cases we significantly outperform existing results.
This may prompt more study into larger and more complex MOIP problems, problems which until now may have been impractical to solve.

Two of our new algorithms utilise the synergistic theory we present.
One of these, SPREAD, significantly outperforms all other algorithms on the larger test cases, including the other synergistic algorithm CLUSTER.
The difference between SPREAD and CLUSTER is how permutations are chosen.
It may be useful to further study how this choice may affect the running time of our algorithms, especially as it relates to specific MOIP problems.
The extension of EPP to projections to two or more dimensions may also prove useful in scenarios where many threads are available.

The publication of our implementations as well as our algorithms allows the easier comparison of the running time of exact MOIP algorithms, and will hopefully spur further research and development in this field.

\section*{Acknowledgements}
This study was supported by the Australian Research Council under the Discovery Projects funding scheme (project DP140104246).

\appendix

\section{Examples}
The first two examples are detailed walk throughs of solving constrained lexicographic problems where the permutation is the identity permutation.
\begin{example}[Calculating $\OIP_{(1,2,3)}^3(2, (52))$]\label{ex:lip}
Consider the following set of objective vectors%Example \ref{ex:moip}.
\[
\begin{array}{ccl}
  f_1 & f_2 & f_3 \\ \hline
  (50 & 24 & 44) \\
  (46 & 41 & 41) \\
  (37 & 46 & 37) \\
  (37 & 44 & 42) \\
  (32 & 39 & 54).
\end{array}
\]
The value $(52)$ in the definition of the problem says that we are only interested in objective vectors which satisfy $f_3 \leq 52$.
This immediately rules out $(32,39,54)$, and we no longer use this objective vector for any domination tests, leaving us with the following.
\[
\begin{array}{cc|c}
  f_1 & f_2 & f_3 \\ \hline
  (50 & 24 & 44) \\
  (46 & 41 & 41) \\
  (37 & 46 & 37) \\
  (37 & 44 & 42)
\end{array}
\]
Next, the $2$ indicates that we want to discard any objective vector which is dominated in its first two objective values by some other objective vector which we have not discarded.
This is represented in the table by the columns to the left of the vertical line.
We see that $(37, 46, 37)$ is dominated over the first two objectives by $(37, 44, 42)$.
Even though $37 < 42$, we discard $(37, 46, 37)$ as we only consider the first two objectives.
All remaining objective vectors are not dominated in their first two objective values, so we are done and the non-dominated objective vectors for $\LIP^n(2, (52)$ are $\{ (50, 24, 44), (46,41,41), (37, 44, 42) \}$.
\end{example}

\begin{example}[Calculating $\OIP_{(1,2,3)}^3(1, (48, 43))$]\label{ex:lip2}
  Again we are working from 
\[
\begin{array}{ccl}
  f_1 & f_2 & f_3 \\ \hline
  (50 & 24 & 44) \\
  (46 & 41 & 41) \\
  (37 & 46 & 37) \\
  (37 & 44 & 42) \\
  (32 & 39 & 54).
\end{array}
\]
  We can immediately discard $(50, 24, 44)$ and $(32, 39, 54)$ from the given upper bounds $(48, 43)$, leaving
\[
\begin{array}{c|cc}
  f_1 & f_2 & f_3 \\ \hline
  (46 & 41 & 41) \\
  (37 & 46 & 37) \\
  (37 & 44 & 42).
\end{array}
\]
  We next consider dominance in the first objective only, letting us discard $(46, 41, 41)$.
  This leaves us with
\[
\begin{array}{c|cc}
  f_1 & f_2 & f_3 \\ \hline
  (37 & 46 & 37) \\
  (37 & 44 & 42).
\end{array}
\]
  These are equal in their first objective, so neither dominates the other.
  We then consider the final two objective functions in lexicographic order.
  That is, we consider $f_2$ before $f_3$ and so-on.
  As $44 < 46$, we discard $(37, 46, 37)$ and the set of non-dominated objective vectors for $\OIP_{(1,2,3)}^3(1, (48, 43))$ is $\{ (37, 44, 42)\}$.
\end{example}

We now show how different permutations $s$ affect the ordered variants, ordered lexicographic problems.
\begin{example}[Calculating $\OIP^3_{(2,1,3)}(1, (48,43))$]\label{ex:lips}
  We work from the same initial objective vectors set as the earlier examples.
\[
\begin{array}{ccc}
  f_1 & f_2 & f_3 \\ \hline
  (50 & 24 & 44) \\
  (46 & 41 & 41) \\
  (37 & 46 & 37) \\
  (37 & 44 & 42) \\
  (32 & 39 & 54).
\end{array}
\]
To aid our understanding of how the permutation effects the problem, however, we rearrange the columns according to $s$ to give
\[
\begin{array}{ccc}
  f_2 & f_1 & f_3 \\ \hline
  (24 & 50 & 44) \\
  (41 & 46 & 41) \\
  (46 & 37 & 37) \\
  (44 & 37 & 42) \\
  (39 & 32 & 54).
\end{array}
\]
  We now demonstrate which of these correspond to solutions of $\OIP^3_{(2,1,3)}(1, (48, 43))$.
  First, we discard objective vectors that break the given bounds.
  As $s(2) = 1$, we discard objective vectors with $f_1 > 48$.
  The $48$ refers to an upper bound on $f_1$ due to the permutation $s$.
  This causes us to discard $(50, 24, 44)$ (which appears as $(24, 50, 44)$ in the above table as we re-ordered the columns).
  Also, as $s(3) = 3$, we discard objective vectors with $f_3 > 52$.
  That is, we once again discard $(32,39,54)$.
\[
\begin{array}{c|cc}
  f_2 & f_1 & f_3 \\ \hline
  (41 & 46 & 41) \\
  (46 & 37 & 37) \\
  (44 & 37 & 42) 
\end{array}
\]
  We now consider dominance on objective $f_2$.
  We use $f_2$ as $s(1) = 2$, and see that $(46, 41,41 )$ is the unique solution to attain a minimum on $f_2$.
  Our set of non-dominated objective vectors is $\{ (46, 41, 41)\}$.
\end{example}

\begin{example}[Calculating $\OIP^3_{(1,3,2)}(1, (51, 50))$]\label{ex:lips2}
  We work from the same initial objective vectors, and again permute the columns according to $s$.
\[
\begin{array}{ccc}
  f_1 & f_3 & f_2 \\ \hline
  (50 & 44 & 24) \\
  (46 & 41 & 41) \\
  (37 & 37 & 46) \\
  (37 & 42 & 44) \\
  (32 & 54 & 39)
\end{array}
\]
  As $s(2) = 3$, we discard objective vectors that don't satisfy $f_3 < 51$, that is $(32, 39, 54)$.
  And as $s(3) = 2$, we discard objective vectors that don't satisfy $f_2 < 50$, but all objective vectors satisfy this bound.
  Next we consider dominance across objective $s(1) = 1$.
\[
\begin{array}{c|cc}
  f_1 & f_3 & f_2 \\ \hline
  (50 & 44 & 24) \\
  (46 & 41 & 41) \\
  (37 & 37 & 46) \\
  (37 & 42 & 44)
\end{array}
\]
  Once again we are left with $(37, 46, 36)$ and $(37, 44, 42)$, which are equal in their first objective.
\[
\begin{array}{c|cc}
  f_1 & f_3 & f_2 \\ \hline
  (37 & 37 & 46) \\
  (37 & 42 & 44)
\end{array}
\]
  However, we now consider the remaining two objectives in the order prescribed by $s$.
  As $s(2) = 3$, we consider values of $f_3$ next and as $36 < 42$, we discard $(37, 44, 42)$.
  Therefore the set of non-dominated objective vectors for $\OIP^3_{(1,3,2)}(1, (51, 50))$ is $\{ (37, 46, 36) \}$.
\end{example}

\bibliographystyle{wplain}
\bibliography{bibliography}

\end{document}